\documentclass{article}
\usepackage{maa-monthly}

\newcommand{\floor}[1]{\lfloor #1 \rfloor}
\newcommand{\Z}{\mathbb{Z}}
\newcommand{\R}{\mathbb{R}}

\newcommand{\N}{\mathbb{N}}

\allowdisplaybreaks

\newtheorem{Proposition}{Proposition}
\newtheorem{Gem}{Gem}
\newtheorem{Theorem}[Proposition]{Theorem}

\theoremstyle{definition}

\begin{document}

\title{Four quotient set gems}
\markright{Four quotient set gems}

\author{Bryan Brown, Michael Dairyko, Stephan Ramon Garcia,\\ Bob Lutz, and Michael Someck}

\maketitle

\begin{abstract}
Our aim in this note is to present four remarkable facts about quotient sets.
These observations seem to have been overlooked by the \textsc{Monthly}, despite its intense coverage of quotient sets over the years.
\end{abstract}

\section*{Introduction}

	If $A$ is a subset of the natural numbers $\N = \{1,2,\ldots\}$, then we let
	$R(A) = \{a/a' : a,a' \in A\}$ denote the corresponding \emph{quotient set} (sometimes
	called a \emph{ratio set}).  Our aim in this short note is 
	to present four remarkable results which seem to have been overlooked in the 
	\textsc{Monthly}, despite its intense coverage of quotient sets over the years
	\cite{BT, QSDE, QGP, Hedman, Hobby, Nowicki, Starni}.  Some of these results are novel, 
	while others have appeared in print elsewhere but somehow remain largely unknown.
	
	In what follows, we let $A(x) = A \cap [1,x]$ so that $|A(x)|$ denotes the number of elements in 
	$A$ which are $\leq x$.  
	The \emph{lower asymptotic density} of $A$ is the quantity
	\begin{equation*}
		\underline{d}(A) = \liminf_{n\to\infty} \frac{ | A(n)| }{n},
	\end{equation*}
	which satisfies the obvious bounds $0 \leq \underline{d}(A) \leq 1$.  
	We say that $A$ is \emph{fractionally dense} if the closure of $R(A)$ in $\R$
	equals $[0,\infty)$ (i.e., if $R(A)$ is dense in $[0,\infty$)).
	
	Our four gems are as follows.
	\begin{enumerate}\addtolength{\itemsep}{0.5\baselineskip}
		\item[1.] The set of all natural numbers whose base-$b$ representation begins with the digit $1$
			is fractionally dense for $b=2,3,4$, but not for $b \geq 5$.
	
		\item[2.] For each $\delta \in [0,\frac{1}{2})$, there exists a set $A \subset \N$ with
			 $\underline{d}(A) = \delta$ that is not fractionally dense.
			On the other hand, if $\underline{d}(A) \geq \frac{1}{2}$, then $A$ 
			must be fractionally dense \cite{ST}.
		
		\item[3.] One can partition $\N$ into three sets, each of which is not fractionally
			dense.  However, such a partition is impossible using only two sets \cite{BST}.
		
		\item[4.] There are subsets of $\N$ which contain arbitrarily long arithmetic progressions, yet that
			are not fractionally dense.  On the other hand, there exist fractionally dense sets that
			have no arithmetic progressions of length $\geq 3$.
	\end{enumerate}

\section*{Base-$b$ representations}

	In \cite[Example 19]{QSDE}, it was shown that the set
	\begin{equation*}
		A = \{1\} \cup \{10,11,12,13,14,15,16,17,18,19\} \cup \{100,101,\ldots\} \cup \cdots
	\end{equation*}
	of all natural numbers whose base-$10$ representation begins with the digit $1$ is not fractionally dense.
	This occurs despite the fact that $\underline{d}(A) = \frac{1}{9}$,
	so that a positive proportion of the natural numbers belongs to $A$.  
	The consideration of other bases reveals the following gem.
		
	\begin{Gem}\label{G1}
		The set of all natural numbers whose base-$b$ representation begins with the digit $1$
		is fractionally dense for $b=2,3,4$, but not for $b \geq 5$.
	\end{Gem}

	To show this, we require the following more general result.	

	\begin{Proposition}\label{PropositionAB}
		Let $1 < a \leq b$.  The set
		\begin{equation}\label{eq:AB}
			A = \bigcup_{k=0}^{\infty} [b^k,ab^k) \cap \N
		\end{equation}
		is fractionally dense if and only if $b \leq a^2$.  Moreover, we also have
		\begin{equation*}
			\underline{d}(A) = \frac{a-1}{b-1}.
		\end{equation*}
	\end{Proposition}
	
	\begin{proof}
		We first compute $\underline{d}(A)$.
		Since the counting function $|A(x)|$ is nondecreasing on each interval 
		of the form $[b^k,a b^k)$ and constant on
		each interval of the form $[a b^k, b^{k+1})$, it follows that
		\begin{align}
			\liminf_{n\to\infty} \frac{ |A(n)|}{n} 
			&= \liminf_{n\to\infty} \frac{1}{b^n} \sum_{k=1}^{n-1} \big|[b^k,a b^k)\cap \N \big| \nonumber \\
			&= \lim_{n\to\infty} \frac{1}{b^n} \sum_{k=1}^{n-1} (a-1) b^k \label{eq:UFT}\\
			&=\lim_{n\to\infty} \frac{(a-1) (b^n-1)}{b^n(b - 1)} \nonumber \\
			&=\frac{a-1}{b-1}. \nonumber
		\end{align}	
		Note that in order to obtain \eqref{eq:UFT} we used the fact that the difference between
		$ab^k - b^k$ and $|[b^k,a b^k)\cap \N |$ is $\leq 1$.
		Having computed $\underline{d}(A)$, we now turn our attention to fractional density.
		\medskip

		\noindent\textsc{Case 1}:  Suppose that $a^2 < b$.
		By construction, each quotient of elements of $A$ belongs to an interval of the form
		$I_{\ell} = (a^{-1} b^{\ell} , a b^{\ell})$ for some integer $\ell$.  If $j < k$, 
		then $a^2 < b^{k-j}$ whence
		$a b^j < a^{-1} b^k$ so that every element of $I_j$ 
		is strictly less than every element of $I_k$.
		Therefore $R(A)$ contains no elements in any interval of the form
		$[a b^{\ell},a^{-1}b^{\ell+1}]$, each of which is nonempty since $a^2 < b$.
		Thus $A$ is not fractionally dense.
		\medskip

		\noindent\textsc{Case 2}: 
		Now let $b \leq a^2$, noting that 
		\begin{equation*}
			(0,\infty)  
			= \bigcup_{j \in\Z} \big[ \tfrac{b^j}{a}, b^j\big) \cup [b^j,a  b^j).
		\end{equation*}		
		Suppose that $\xi$ belongs to an interval of the form $[b^j,ab^j)$ for some integer $j$.
		Given $\epsilon >0$, let $k$ be so large that $1 < b^k \epsilon$ and observe that
		\begin{equation*}
			b^{j+k}\leq b^k\xi< ab^{j+k}.
		\end{equation*}
		Let $\ell$ be the unique natural number satisfying
		\begin{equation}\label{eq:bjk}
			b^{j+k} + \ell \leq b^k \xi < b^{j+k} + \ell + 1
		\end{equation}
		and
		\begin{equation}\label{eq:lab}
			0 \leq \ell \leq (a-1)b^{j+k}-1.
		\end{equation}
		From \eqref{eq:bjk} we find that
		\begin{equation*}
			0 \leq b^k \xi - (b^{j+k}+\ell) < 1,
		\end{equation*}
		from which it follows that
		\begin{equation*}
			0 \leq \xi - \frac{b^{j+k}+\ell}{b^k} < \epsilon.
		\end{equation*}
		Since \eqref{eq:lab} ensures that $b^{j+k} + \ell$ belongs to the interval $[b^{j+k} ,a b^{j+k})$,
		we conclude that 
		$(b^{j+k} + \ell)/b^k$ belongs to $R(A)$.
		A similar argument applies if $\xi$ belongs to an interval of the form 
		$[ \tfrac{b^j}{a}, b^j)$.  Putting this all together, we conclude that $A$ is fractionally dense.
	\end{proof}

	To see that Gem \ref{G1} follows from the preceding proposition, observe that
	if $a =2$ and $b \geq 2$ is an integer, then the set $A$ defined by 
	\eqref{eq:AB} is precisely the set of all natural numbers whose base-$b$ representation
	begins with the digit $1$.  The inequality $b \leq a^2$ holds precisely for the bases
	$b = 2,3,4$ and fails for all $b \geq 5$.

\section*{Critical density}	

	Having seen that there exist sets that are not fractionally dense, yet whose lower asymptotic density is positive,
	it is natural to ask whether there exists a critical value $0 < \kappa \leq 1$ such that $\kappa \leq \underline{d}(A)$ ensures that
	$A$ is fractionally dense.  The following gem establishes the existence of such a critical density, namely $\kappa = \frac{1}{2}$.

	\begin{Gem}
		If $\underline{d}(A) \geq \frac{1}{2}$, then $A$ 
		is fractionally dense.
		On the other hand, if $0 \leq \delta < \frac{1}{2}$, then
		there exists a subset $A \subset \N$ with $\underline{d}(A) = \delta$
		that is not fractionally dense.
	\end{Gem}

	Establishing this result will take some work, although we are now in a position to prove the second statement
	(originally obtained by Strauch and T\'oth \cite[Thm.~1]{ST} and by Harman \cite[p.~167]{Harman} using slightly different methods). 
	If $0 \leq \delta < \frac{1}{2}$, then
	we may write $\delta = \frac{1}{2+\epsilon}$ where $\epsilon > 0$.   Now let
	$a = 1 + \frac{\epsilon}{2}$ and $b = 1 + \epsilon + \frac{1}{2}\epsilon^2$
	so that $1 < a^2 < b$.  For these parameters, Proposition \ref{PropositionAB} tells us that the set $A$ defined by \eqref{eq:AB}
	satisfies $\underline{d}(A) = \delta$ and fails to be fractionally dense.
	If $\delta = 0$, then we note that the set $A = \{2^n:n\in \N\}$ is not fractionally dense and has $\underline{d}(A)=0$.
	
	Thus we can construct
	sets having lower asymptotic density arbitrarily close to $\frac{1}{2}$, yet which fail to be fractionally dense.
	On the other hand, it is clear that fractionally dense sets with $\underline{d}(A) = \frac{1}{2}$ exist.
	Indeed, one such example is $A = \{2,4,6,\ldots\}$.  However, the question of whether a non-fractionally
	dense set can have lower asymptotic density \emph{equal} to the critical density $\frac{1}{2}$ is much more difficult.

	In the late 1960s, 
	{\v{S}}al{\'a}t proposed an example of a set $A \subset \N$ which is not fractionally
	dense and such that $\underline{d}(A) = \frac{1}{2}$ \cite[p.~278]{Salat}.  
	However, {\v{S}}al{\'a}t's example was flawed, as pointed out in the associated corrigendum \cite{SalatC}.
	In 1998, Strauch and T\'oth established a more general result \cite[Thm.~2]{ST} which
	implies that a set satisfying $\underline{d}(A) \geq \frac{1}{2}$ must be fractionally dense
	(fortunately, the lengthy corrigendum \cite{STC} to this paper does not affect
	the result in question).  For $\underline{d}(A) > \frac{1}{2}$, the fractional density of $A$ 
	also follows from a sophisticated theorem in metric number theory \cite[Thm.~6.6]{Harman}
	(although Harman was kind enough to show us an elementary proof in this case).  
	The proof below, which covers the critical case $\underline{d}(A) = \frac{1}{2}$, 
	is essentially due to Strauch and T\'oth.

\begin{Theorem}\label{TheoremST}
	If $\underline{d}(A) \geq \frac{1}{2}$, then $A$ is fractionally dense.
\end{Theorem}

\begin{proof}
	Let $\underline{d}(A) \geq \frac{1}{2}$ and suppose toward a contradiction
	that $0<\alpha<\beta\leq 1$ and $R(A) \cap (\alpha,\beta) = \varnothing$.
	Noting that $A$ must be infinite, we enumerate the elements of $A$
	in increasing order $a_1 < a_2 < \cdots$.  Let $k$ be a natural number which is
	so large that $k\alpha>1$ and let $0 < \theta <1$.
	For all $m$ and $n$ in $\N$, let
	\begin{equation*}
		J_m^n=\big(\alpha a_{k(\lfloor\theta n\rfloor+m)},\alpha (a_{k(\lfloor\theta n\rfloor+m)}+k)\big).
	\end{equation*}
	
	For each $n$ in $\N$, we claim that the intervals
	\begin{equation}
		\label{eq-intervals}
		J_0^n,\, J_1^n,\ldots,\, J_{n-\lfloor\theta n\rfloor-1}^n,\, (\alpha a_{kn},\beta a_{kn}),
	\end{equation}
	used below in a delicate counting argument, are pairwise disjoint.  Indeed, since the $a_i$ are integers it follows that
	\begin{equation*}
		\alpha (a_{k(\lfloor\theta n\rfloor+m)}+k)\leq 
		\alpha (a_{k(\lfloor\theta n\rfloor+m)+k}) = \alpha(a_{k(\lfloor\theta n\rfloor+m+1)}),
	\end{equation*}
	so that the right endpoint of $J_m^n$ is at most the left endpoint of $J_{m+1}^n$. A similar argument shows that the right endpoint of $J_{n-\lfloor\theta n\rfloor-1}^n$ is at most $\alpha a_{kn}$.
	
	Next, let $n$ be so large that $\floor{\theta n}\geq \frac{\alpha}{\beta-\alpha}$. We claim that
	\begin{equation}\label{eq-containment}
		J_m^n\subseteq \big(\alpha a_{k(\lfloor\theta n\rfloor+m)}, \beta a_{k(\lfloor\theta n\rfloor+m)} \big)
	\end{equation}
	for each $m=0,1,\ldots, n-\lfloor\theta n\rfloor-1$.  To see this, note that
	\begin{equation*}
		\frac{k\alpha}{\beta - \alpha} \leq k\floor{\theta n} \leq a_{k \floor{\theta n}} < a_{k \floor{\theta n}+m},
	\end{equation*}
	which implies that $\alpha (a_{k(\lfloor\theta n\rfloor+m)}+k)\leq \beta a_{k(\lfloor\theta n\rfloor+m)}$
	so that \eqref{eq-containment} holds.
	
	Since $(\alpha a_n,\beta a_n)\cap A = \varnothing$ by assumption, it follows now that the intervals \eqref{eq-intervals} 
	are contained in the complement $[0,\infty) \backslash A$.  Letting $B = \N\setminus A$, a na\"ive counting argument gives
	\begin{equation}  \label{eq-mainineq}
		|B(\beta a_{kn})|
		\geq \big((\beta-\alpha)a_{kn}-1\big)+(n-\lfloor \theta n\rfloor)(k\alpha-1).\nonumber
	\end{equation}
	Dividing through by $\beta a_{kn}$ and taking lim sups yields
	\begin{align*}
		1 - \underline{d}(A)
		&= \limsup_{n\to\infty} \frac{|B(n)|}{n} \\
		&\geq \limsup_{n\to\infty}  \frac{|B (\beta a_{kn} )| }{ \beta a_{kn} }\\
		&\geq\limsup_{n\to\infty}\left( \frac{(\beta-\alpha)a_{kn}-1}{\beta a_{kn}}+\frac{(n-\lfloor \theta n\rfloor)(k\alpha-1)}{\beta a_{kn}}\right) \\
		&\geq \frac{\beta-\alpha}{\beta}+\liminf_{n\to\infty} \frac{(n-\lfloor \theta n\rfloor)(k\alpha-1)}{\beta a_{kn}} \\
		&\geq 1-\frac{\alpha}{\beta}+(1-\theta)\liminf_{n\to\infty} \frac{\alpha kn -n}{\beta a_{kn}} \\
		&\geq  1-\frac{\alpha}{\beta}+(1-\theta)\left( \frac{\alpha}{\beta} \liminf_{n\to\infty} \frac{kn}{a_{kn}} 
		- \frac{1}{\beta k}  \limsup_{n\to\infty} \frac{kn}{a_{kn}}\right)\\
		&\geq  1-\frac{\alpha}{\beta}+(1-\theta)\left( \frac{\alpha}{\beta} \underline{d}(A) 
		- \frac{1}{\beta k}\right).
	\end{align*}
	Taking the limit as $\theta\to 0$ and $k\to\infty$ and rearranging gives
	\begin{equation}\label{eq-dab}
		\left(1+\frac{\alpha}{\beta}\right) \underline{d}(A) \leq \frac{\alpha}{\beta}.
	\end{equation}
	If $\underline{d}(A) \geq \frac{1}{2}$, then the preceding inequality implies that $\beta \leq \alpha$,
	a contradiction.
\end{proof}

The astute reader will note that we have actually shown that if 
$(\alpha,\beta)\cap R(A)=\varnothing$, then \eqref{eq-dab} must hold.
In fact, with only a little more work one can show that
\begin{equation*}
	\underline{d}(A)\leq \frac{\alpha}{\beta} \min\big\{\overline{d}(A),1-\overline{d}(A)\big\}
\end{equation*}
and $\overline{d}(A) \leq 1 - (\beta - \alpha)$,
where 
\begin{equation*}
	\overline{d}(A) = \limsup_{n\to\infty} \frac{|A(n)|}{n}
\end{equation*}
denotes the \emph{upper asymptotic density} of $A$ \cite[Thm.~2]{ST}.
Moreover, it is also known that $\underline{d}(A) + \overline{d}(A) \geq 1$ implies that
$A$ is fractionally dense \cite[p.~71]{ST}.

\section*{Partitions of $\N$}
	Clearly $\N$ itself is fractionally dense since $R(\N) = \mathbb{Q} \cap (0,\infty)$, the set
	of positive rational numbers.  An interesting question now presents itself.  If $\N$ is partitioned
	into a finite number of disjoint subsets, does one of these subsets have to be fractionally dense?  
	The following result gives a complete answer to this problem.

	\begin{Gem}\label{G3}
		One can partition $\N$ into three sets, each of which is not fractionally
		dense.  However, such a partition is impossible using only two sets.
	\end{Gem}

	This gem is due to Bukor, \v{S}al\'{a}t, and T\'oth \cite{BST}.  These three, along with
	P.~Erd\H{o}s, later generalized the second statement by showing
	that if the set $A$ is presented as an increasing sequence $a_1 < a_2 <\cdots$ satisfying
	$\lim_{n\to\infty} a_{n+1}/a_n = 1$, then for each $B \subseteq A$, either
	$B$ or $A\backslash B$ is fractionally dense \cite{Erdos}.
	
	The proof of our third gem is contained in the following two results.
	
	\begin{Proposition}\label{PropositionClean}
		There exist disjoint sets $A,B,C \subset \N$, none of which are fractionally dense,
		such that $\N = A\cup B \cup C$.
	\end{Proposition}
	
	\begin{proof}
		Let
		\begin{equation*}
			A = \bigcup_{k=0}^{\infty} \big[5^k, 2\cdot 5^k\big) \cap \N, \quad
			B = \bigcup_{k=0}^{\infty} \big[2\cdot 5^k,3\cdot 5^k\big) \cap \N, \quad
			C = \bigcup_{k=0}^{\infty} \big[3\cdot 5^k, 5\cdot 5^k\big) \cap \N,
		\end{equation*}
		so that $A$, $B$, and $C$
		consist of those natural numbers whose base-$5$ expansions begin, respectively, with
		$1$, with $2$, and with $3$ or $4$.
		We consider only $C$ here, for the remaining two cases are similar
		(in fact, the set $A$ is already covered by Proposition \ref{PropositionAB}).
		Observe that each quotient of elements of $C$ is contained in an interval of the form
		$I_{\ell} = \left( \frac{3}{5} 5^{\ell}, \frac{5}{3} 5^{\ell} \right)$ for some integer $\ell$.
		If $j < k$, then every element of $I_j$ is strictly less than every element of $I_k$ since
		$\frac{5}{3}\cdot 5^j < \frac{3}{5}\cdot 5^k$ holds if and only if
		$\frac{25}{9} < 5^{k-j}$.  Thus $R(C) \cap [ \frac{5}{3} 5^{\ell}, \frac{3}{5} 5^{\ell+1} ] = \varnothing$
		for any integer $\ell$.
	\end{proof}

	\begin{Theorem}\label{TheoremBST}
		If $A$ and $B$ are disjoint sets such that $\N = A \cup B$, then at least one of $A,B$ is 
		fractionally dense.
	\end{Theorem}

	\begin{proof}
		Without loss of generality, we may assume that both $A$ and $B$ are infinite.
		Suppose toward a contradiction that neither $A$ nor $B$ is fractionally dense.
		Thus there exists $\alpha,\beta>1$ and $\epsilon > 0$ such that
		$(\alpha-\epsilon, \alpha+\epsilon) \cap R(A)$ and 
		$(\beta-\epsilon,\beta+\epsilon) \cap R(B)$ are both empty.
		Now let $n_0 \in \N$ be such that
		\begin{equation*}
			\frac{1+\alpha +\beta +2\alpha \beta}{n_0} < \epsilon.
		\end{equation*}
		
		Since $A$ and $B$ are both infinite, there exists
		$n > \alpha \beta(n_0 + 1)$ such that $n \in A$ and $n+1 \in B$.
		If $s = \lfloor \frac{n}{\alpha \beta} \rfloor - 1$ belongs to $A$, then setting
		$t = \floor{\alpha s}$ yields
		\begin{equation}\label{eq-Before}
			\left| \frac{t}{s} - \alpha \right| 
			= \left| \frac{ \floor{\alpha s}-\alpha s}{s}\right| 
			< \frac{1}{s} \leq \frac{1}{n_0} < \epsilon.
		\end{equation}
		Since $(\alpha-\epsilon, \alpha+\epsilon) \cap R(A) = \varnothing$,
		we conclude that $t$ belongs to $B$.  Now observe that 
		\begin{align*}
			\left| \frac{n+1}{t} - \beta  \right|
			&=  \frac{n+1 - \beta  \floor{\alpha s} }{t}  \\
			&< \frac{n+1-\beta (\alpha s-1)}{t} \\
			&= \frac{1+\beta + \alpha \beta +n - \alpha \beta \floor{ \frac{n}{\alpha \beta } } }{t} \\
			&< \frac{1+\beta +2\alpha \beta }{t} \\
			&\leq \frac{1+\beta +2\alpha \beta }{n_0} \\
			&< \epsilon. 
		\end{align*}
		Hence $\frac{n+1}{t}$ belongs to $(\beta-\epsilon, \beta+\epsilon) \cap R(B)$, which is a contradiction.
		Therefore it must be the case that $s$ belongs to $B$.
	
		Assuming now that $s = \lfloor \frac{n}{\alpha \beta} \rfloor - 1$ belongs to $B$, we now let
		$t = \floor{\beta s}$.  Proceeding as in \eqref{eq-Before}, one can show that 
		$| \frac{t}{s}-\beta| < \epsilon$ and hence that $t$ belongs to $A$.  Moreover,	
		\begin{align*}
			\left| \frac{n}{t} - \alpha \right| 
			&= \left| \frac{n - \alpha \floor{\beta s}}{t} \right| \\
			&< \frac{n - \alpha (\beta s-1)}{t} \\
			&= \frac{\alpha + n - \alpha \beta s}{t} \\
			&= \frac{\alpha + n - \alpha \beta (\floor{ \frac{n}{\alpha \beta} } - 1)}{t} \\
			&=  \frac{\alpha +  \alpha \beta + (n - \alpha \beta \floor{ \frac{n}{\alpha \beta} } )}{t}  \\
			&< \frac{\alpha +2\alpha \beta}{n_0} \\
			&< \epsilon.
		\end{align*}
		Thus $\frac{n}{t}$ belongs to $(\alpha - \epsilon, \alpha + \epsilon) \cap R(A)$,
		a contradiction which implies that $s$ belongs to neither $A$ nor $B$, an absurdity
		since $\N = A \cup B$.
	\end{proof}

	Although one might be tempted to postulate a relationship between Theorems \ref{TheoremST}
	and \ref{TheoremBST}, lower asymptotic density has no bearing on the
	preceding result since it is possible to partition $\N$ into two subsets both having lower asymptotic density zero.

	\begin{Proposition}
		There exist disjoint sets $A,B \subset \N$ such that $\N = A \cup B$,
		$\underline{d}(A) = \underline{d}(B) = 0$, and $\overline{d}(A) = \overline{d}(B) = 1$. 
	\end{Proposition}
	
	\begin{proof}
		We require the Stolz-Ces\`aro Theorem \cite{Gelca}, which tells us that if $x_n$ and $y_n$ 
		are two increasing and unbounded sequences of real numbers, then
		\begin{equation*}
			\lim_{n\to\infty} \frac{x_{n+1} - x_n}{y_{n+1} - y_n} = L
			\quad\Longrightarrow\quad \lim_{n\to\infty} \frac{x_n}{y_n} = L.
		\end{equation*}
		Now put the first $1!$ natural numbers into $A$,
		then the next $2!$ into $B$, the next $3!$ into $A$, and so forth, yielding the sets
		\begin{equation*}
			A=\{1,4,5,6,7,8,9,34,...\},\qquad B=\{2,3,10,11,12,\ldots,32,33,154,155,\ldots\}.
		\end{equation*}
		By construction, $A \cap B = \varnothing$ and $\N = A \cup B$.
		Let $x_n = \sum_{k=1}^{2n-1} k!$ and $y_n = \sum_{k=1}^{2n} k!$, observing that
		$|A(x_n)| = |A(y_n)| = \sum_{k=1}^n (2k-1)!$.  Since
		\begin{equation*}
			\frac{ |A(y_{n+1})| - |A(y_n)|}{y_{n+1} - y_n} = \frac{ (2n+1)! }{ (2n+2)! + (2n+1)!} 
			= \frac{1}{2n+3} \to 0
		\end{equation*}
		and
		\begin{equation*}
			\frac{ |A(x_{n+1})| - |A(x_n)|}{x_{n+1} - x_n} = \frac{ (2n+1)!}{(2n+1)! + (2n)!} 
			= \frac{1}{1+ \frac{1}{2n+1}} \to 1,
		\end{equation*}
		it follows from the Stolz-Ces\`aro Theorem that $\underline{d}(A) = 0$ and $\overline{d}(A) = 1$.
		An analogous argument using $x_n = \sum_{k=1}^{2n} k!$ and $y_n = \sum_{k=1}^{2n+1} k!$ 
		now reveals that $\underline{d}(B) = 0$ and $\overline{d}(B) = 1$ as well.
	\end{proof}

\section*{Arithmetic progressions}

	Our final gem concerns arithmetic progressions of natural numbers.
	Recall that an \emph{arithmetic progression} with common difference $b$ and length $n$
	is a sequence of the form $a, a+b, a+2b,\ldots, a+(n-1)b$.  Subsets of the natural numbers 
	that contain arbitrarily long arithmetic progressions are often thought of as
	being ``thick'' or ``dense'' in some qualitative sense.  
	
	\begin{Gem}\label{G4}
		There exists a subset of $\N$ containing arbitrarily long arithmetic progressions, yet that
		is not fractionally dense.  On the other hand, there exists a fractionally dense set that
		has no arithmetic progressions of length $\geq 3$.
	\end{Gem}

The first statement of Gem \ref{G4} has already been proven.  Indeed, Proposition \ref{PropositionAB}
produces sets that are not fractionally dense and that 
contain arbitrarily long blocks of consecutive natural numbers.  We therefore need
only produce a fractionally dense set having no arithmetic progressions of length three
(see also  \cite[Thm.~2]{Hedman}).

\begin{Proposition}
	The set $A = \{ 2^j : j \geq 2\} \cup \{ 3^k : k \geq 2\}$ is fractionally dense
	and has no arithmetic progressions of length three.
\end{Proposition}

\begin{proof}
	First recall that Kronecker's approximation 
	theorem \cite[Thm.~440]{Hardy} asserts that 
	if $\beta > 0$ is irrational, $\alpha \in \R$, and $\delta > 0$, then there exist
	$n,m \in \N$ such that $| n \beta - \alpha - m| < \delta$.
	Let $\xi,\epsilon >0$ and note that $\beta = \log_2 3 >0$ is irrational.
	By the continuity of $f(x) = 2^x$ at $\log_2 \xi$, there exists $\delta > 0$ such that
	\begin{equation}\label{eq-EAED}
		| \log_2 x - \log_2 \xi| < \delta \quad\implies\quad | x - \xi | < \epsilon.
	\end{equation}		
	Kronecker's theorem with $\beta = \log_2 3$
	and $\alpha = \log_2 \xi$ now yields $n,m \in \N$ so that
	\begin{equation*}
		\left| \log_2 \left(\frac{3^n}{2^m}\right) - \log_2 \xi \right| \,\,=\,\,
		 | n  \log_2 3 - \log_2 \xi - m| < \delta.
	\end{equation*}
	In light of \eqref{eq-EAED}, it follows that $|3^n /2^m - \xi| < \epsilon$, whence 
	$A$ is fractionally dense.
	
	We now show that $A$ contains no arithmetic progressions of length three.
	Assume toward a contradiction that such an arithmetic progression exists.
	By the definition of $A$, the first term in this progression is either a power of $2$
	or a power of $3$.  We consider both cases separately, letting $b$ denote the common difference
	in each progression.
	\medskip
	
	\noindent\textsc{Case 1}:	
	Suppose that $2^j$, $2^j+b$, and $2^j+2b$ belong to $A$.
	Since $2^j+2b$ is even and belongs to $A$, 
	it must be of the form $2^k$ for some $k > j$ whence $b = 2^{k-1} - 2^{j-1}$ is even since
	$j \geq 2$.  Thus $2^j+b$ must also be of the form $2^{\ell}$ for some $\ell > j$ so that
	\begin{equation*}
		2^{\ell} = 2^j+b = 2^j + (2^{k-1} - 2^{j-1}) = 2^{j-1}(2^{k-j} +1).
	\end{equation*}
	Upon dividing the preceding by $2^{j-1}$ we obtain a contradiction.
	\medskip
	
	\noindent\textsc{Case 2}:	
	Now suppose that $3^j$, $3^j+b$, $3^j + 2b$
	is an arithmetic progression in $A$ of length three that begins with $3^j$.  In this case
	$3^j+2b$ is odd and hence must be of the form $3^k$ for some $k > j$.  Therefore
	\begin{equation*}
		b = \frac{3^k - 3^j}{2} = 3^j\frac{(3^{k-j}-1)}{2},
	\end{equation*}
	so that the second term $3^j+b$ in our progression is divisible by $3^j$.  Thus 
	$3^j + b = 3^{\ell}$ for some $\ell > j$, from which it follows that
	\begin{equation*}
		3^{\ell} = 3^j + b = 3^j + 3^j\frac{(3^{k-j}-1)}{2}.
	\end{equation*}
	Since this implies that $2\cdot 3^{\ell-j} = 2 + (3^{k-j}-1)$,
	which is inconsistent modulo $3$, we conclude that $A$
	has no arithmetic progressions of length three.
\end{proof}

\begin{acknowledgment}{Acknowledgment.}
	Partially supported by National Science Foundation Grant DMS-1001614.
\end{acknowledgment}

\bibliographystyle{monthly}

\bibliography{4QSG}
%

\begin{biog}

\item[Bryan Brown] (bryanbrown740@gmail.com) is a junior at Pomona College majoring in mathematics and minoring in computer science.

\item[Michael Dairyko] (mdairyko@iastate.edu)
	graduated from Pomona College in 2013 with a degree in Mathematics and is currently a Ph.D.~student at Iowa State University. 
	He hopes to become a professor at a small liberal arts where the climate is warm. When not ``mathing,'' 
	he enjoys the outdoors, socializing with friends, and playing waterpolo.  
	\begin{affil}
		Department of Mathematics,
		Iowa State University,
		396 Carver Hall,
		Ames, IA 50011.
	\end{affil}

\item[Stephan Ramon Garcia] (Stephan.Garcia@pomona.edu) 
	grew up in San Jos\'e, California before attending
	U.C.~Berkeley for his B.A. and Ph.D.  After graduating, he worked at U.C.~Santa Barbara 
	for several years before moving to Pomona College in 2006.
	He has earned three NSF research grants and five teaching awards from three different institutions.
	He is the author of almost fifty research articles in operator theory, complex analysis, matrix analysis, and number theory.
	\url{http://pages.pomona.edu/~sg064747}
	\begin{affil}
		Department of Mathematics,
		Pomona College,
		Claremont, California,
		91711, USA	
	\end{affil}

\item[Bob Lutz] (boblutz@umich.edu)
	is a recent graduate from Pomona College and a Ph.D. student in mathematics at the University of Michigan. 
	An ardent gourmand and sometimes-author, he aspires to teach at a liberal arts college one day, 
	but is tempted by the more stable prospect of a career in food blogging.
	\begin{affil}
		Department of Mathematics, University of Michigan,
		2074 East Hall,
		530 Church Street,
		Ann Arbor, MI  48109-1043
	\end{affil}

\item[Michael Someck] (msomeck@gmail.com)
is a junior mathematics and philosophy double major at Pomona College, currently spending a year abroad at the University of Oxford. 
He is interested in the ways in which mathematicians and philosophers often use the same tool, logic, to solve completely different problems. 
In his free time, he enjoys traveling, spending time with friends, 
reading a good book, or relaxing at a beach in his hometown of San Diego.

\end{biog}
\vfill\eject

\end{document}